\documentclass{amsart}

\usepackage{amsmath,amsfonts,amssymb,amscd,verbatim,multicol}
\usepackage[arrow,matrix,cmtip,curve]{xy}
\usepackage{lscape}
\usepackage{fullpage}
 
\usepackage{stmaryrd}
\usepackage{mathrsfs}

\newtheorem{thm}{Theorem}[section]
\newtheorem{cor}[thm]{Corollary}

\newtheorem{prop}[thm]{Proposition}

\newtheorem{defn}[thm]{Definition}
\newtheorem{ex}[thm]{Example}
\newtheorem{rmk}[thm]{Remark}

\newcommand{\cnt}{\mathcal Z}
\newcommand{\kk}{\Bbbk}
\newcommand{\inv}{^{-1}}
\newcommand{\iso}{\cong}
\newcommand{\niso}{\ncong}
\newcommand{\ep}{\varepsilon}

\newcommand{\wa}{A_1^q(\kk)}
\newcommand{\fwa}{A_1(\kk)}
\newcommand{\qp}{\mathcal{O}_q(\kk^2)}
\newcommand{\jp}{\mathcal{J}}

\newcommand{\qm}{\mathcal{O}_q(M_2(\kk))}

\newcommand{\mq}{M_q(2)}
\newcommand{\mj}{M_J(2)}
\newcommand{\jmn}{\mathcal{O}_J(M_2(\kk))}

\DeclareMathOperator{\Aut}{Aut}
\DeclareMathOperator{\id}{id}
\DeclareMathOperator{\SL}{SL}
\DeclareMathOperator{\gldim}{gldim}
\DeclareMathOperator{\GKdim}{GKdim}
\DeclareMathOperator{\Frac}{Frac}

\DeclareMathOperator{\tensor}{\otimes}

\newcommand{\mC}{\mathcal C}
\newcommand{\mG}{\mathcal G}
\newcommand{\mM}{\mathcal M}
\newcommand{\mP}{\mathcal P}
\newcommand{\mR}{\mathcal R}
\newcommand{\mU}{\mathcal U}

\newcommand{\sM}{\mathscr M}
\newcommand{\sP}{\mathscr P}

\newcommand{\NN}{\mathbb N}
\newcommand{\ZZ}{\mathbb Z}

\title{Some Algebras Similar to the $2\times 2$ Jordanian Matrix Algebra}
\author{Jason Gaddis and Kenneth L. Price}

\address{(Gaddis)
Wake Forest University, Department of Mathematics at Statistics, 
P. O. Box 7388, Winston-Salem, NC 27109} 
\email{gaddisjd@wfu.edu}

\address{(Price)
University of Wisconsin - Oshkosh, Department of Mathematics, 
800 Algoma Blvd., Oshkosh, WI 54901} 
\email{pricek@uwosh.edu}

\subjclass[2010]{16S36, 16S85}

\begin{document}

\begin{abstract}
The impetus for this study is the work of Dumas and Rigal on the Jordanian
deformation of the ring of coordinate functions on $2\times 2$ matrices. 
We are also motivated by current interest in birational equivalence of noncommutative rings.
Recognizing the construction of the Jordanian matrix algebra as a skew polynomial ring, 
we construct a family of algebras relative to differential operator rings over a polynomial ring in one
variable which are birationally equivalent to the Weyl algebra over a polynomial ring in two variables.
\end{abstract}

\maketitle

\section{Introduction}

In the study of quantum groups, one wishes to understand algebras
(or their representation theory) which arise through various 
constructions, be it from the Yang-Baxter equation,
Hopf actions, or quantum enveloping algebras.
Our goal is contribute to the understanding of algebras that lie
slightly outside of the study of quantum groups.
That is, to study those algebras which are similar to quantum groups in some sense.
The classic example here is the Weyl algebra, or any of its quantum analogs.
One could approach this problem by looking at PBW deformations of quantum groups,
or by introducing parameters into one of the constructions listed above.

In this work, we will be primarily interested in the $2 \times 2$
quantum matrix algebras, with particular interest in the 
matrix algebra corresponding to the Jordan plane, denoted $\mj$.
We review relevant definitions in Section \ref{sec.matalg}.
Our goal is to put $\mj$ inside a larger class of algebras which
we call \textit{Generalized Jordanian Matrix Algebras} (GJMAs).
These are algebras birationally equivalent to $\mj$ in the sense
that they have the same quotient ring of fractions.
They also maintain a specific involution which can be thought
of as the transposition operator on $2 \times 2$ matrices.
We construct a family of GJMAs as iterated skew polynomial rings.
Much like the universal enveloping algebra of $sl_2$ and its generalizations \cite{Smith},
GJMAs are distinguished by a degree two central element,
which can also be thought of as a generalization of the
determinant in $\mj$.

In Section \ref{sec.skewpolys}, we review the definition of skew polynomial rings,
with an emphasis on those of differential operator type.
We introduce in Section \ref{sec.ispe} a class of algebras which we call
\textit{involutive skew polynomial extensions} (ISPE),
which generalizes the base rings of GJMAs.
These are two-step skew polynomial extensions with an automorphism
fixing the base ring and interchanging the two new generators.
It is shown that under certain conditions an ISPE over a ring $R$
is birationally equivalent to the first Weyl algebra over $R$.

The notion of an ISPE is specialized in Section \ref{sec.pmas} 
to those with base a polynomial ring in one variable.
Such algebras form the base ring of $\mj$
as well as the quantum matrix algebras $\mq$.
We classify these ISPEs in 
Propositions \ref{pma.wa}, \ref{pma.qp}, and \ref{prop.pmapres}.

Section \ref{sec.pmadiff} specializes further
to differential operator rings of the form $\kk[c][a;\delta_1]$.
This class includes the base ring of the Jordanian matrix algebra,
and in general certain elements of this class serve as the base ring of GJMAs.
These algebras are interesting in their 
own right from a representation theory point of view.
In particular, they merge the limited class of
$1$-dimensional representations coming from differential operator rings over $\kk[c]$
with the large class of $n$-dimensional representations corresponding to a
polynomial ring in two variables.

Finally, in Section \ref{sec.gjma}, we introduce and study GJMAs.
If $P$ is an ISPE over $\kk[c]$ of differential operator type 
discussed above and $M=P[b;\sigma_3,\delta_3]$ such that, 
under some localization of $M$,
$\delta_3$ is an inner $\sigma_3$-derivation and $\sigma_3$ an inner automorphism,
then $M$ is a birationally equivalent to $\mj$ (Proposition \ref{prop.gjmain}).
We construct a specific class of such algebras and prove that they are GJMAs (Proposition \ref{gjma.constr})
containing a central element which can be thought of as an analog of the determinant.
Like $\mj$, these algebras are noetherian domains of GK and global dimension four (Proposition \ref{ma.dim}).
The remainder of the work is devoted to a study of the prime ideals of such GJMAs.

\section{Quantum Matrix Algebras}
\label{sec.matalg}

Throughout, $\kk$ is an algebraically closed, 
characteristic zero field and all algebras are $\kk$-algebras. 
Isomorphisms should be read as `isomorphisms as $\kk$-algebras'.
All unadorned tensor products should be regarded as over $\kk$.

For $q \in \kk^\times$, the quantum plane is the $\kk$-algebra $\qp=\kk\langle x_1,x_2 \mid x_1x_2-qx_2x_1\rangle$.
Classically, the \textit{$2 \times 2$ quantum matrix algebra} relative to $q$, sometimes denoted $\qm$,
is the unique $\kk$-algebra on generators $x_{11}$, $x_{12}$, $x_{21}$, $x_{22}$ such that
there exist homomorphisms 
\begin{align*}
\qp &\rightarrow \mq \tensor \qp &\text{ and }& &\qp &\rightarrow \qp \tensor \mq \\
x_i &\mapsto x_{i1} \tensor x_1 + x_{i2} \tensor x_2 & & &x_j &\mapsto x_1 \tensor x_{1j} + x_2 \tensor x_{2j}.
\end{align*}
Making the identifications $b=x_{11}$, $c=x_{22}$, $a=x_{12}$, $d=x_{21}$, we get the relations
\begin{align*}
	ac &= q ca 	& & dc = q cd 	& & da = ad \\
	ba &= q ab 	& & bd = q db 	& & bc = cb + (q-q\inv)ad.
\end{align*}
In this way, $\mq$ may be regarded as a deformation of the coordinate ring of functions on $2 \times 2$ matrices.
For more details on this construction, 
the reader is encouraged to see \cite{BG}, \cite{GW}, and \cite{Manin}.

There are two candidates for deforming the Jordan plane, $\jp = k\langle x_1,x_2 \mid x_2x_1-x_1x_2+x_2^2 \rangle$.
The one presented in \cite{Manin} is not a domain.
We prefer the one presented in \cite{DMMZ} which may be constructed as above by replacing $\qp$ with $\jp$.
This Jordanian matrix algebra, denoted $\mj$ or $\jmn$, 
is the $\kk$-algebra on generators $a,b,c,d$ subject to the following relations.
\begin{align*}
	ac &= ca + c^2  	& &	bc = cb + ca + cd + c^2 \\	
	dc &= cd + c^2 		& &	bd = db + cb + cd - ad + d^2 \\
	da &= ad - cd + ca	& &	ba = ab + cb + cd - ad + a^2.
\end{align*}
The algebra $\mj$ is a domain and can be constructed as a skew polynomial ring.
Dumas and Rigal studied the prime spectrum and automorphisms of this algebra in \cite{DumasRigal},
though the presentation they give is slightly different (but equivalent) to the one
given here.

\begin{rmk}
\label{MJ Alternative}
Letting $u=d-a$, we have the alternate presentation of $\mj$. 
\begin{align*}
ac &= ca+c^2 	& & bc = cb+c(2a+u+c) \\
uc &= cu 		& & bu = ub+u(2a+u+c) \\
ua &= au-cu 		& & ba = (a+c)b+(c-u)a.
\end{align*}
\end{rmk}

\begin{rmk}
The formulas $\tau(a)=d$, $\tau(d)=a$, $\tau(c)=c$, and $\tau(b)=b$
determine either an automorphism $\tau:\mj \rightarrow \mj$ or $\tau:\mq \rightarrow \mq$. 
One should think of this operation as the transposition operator on $2 \times 2$ matrices.
For either $\mq$ or $\mj$ the automorphism group $G$ 
is the semidirect product $H\rtimes \{\tau \}$ for some subgroup $H$ of $G$. 
In the case of $\mq$, we have $H\cong (\kk^{\times})^{3}$ \cite[Theorem 2.3]{AlevChamarie}. 
In the case of $\mj$, $H$ is considerably more complex and the interested reader is referred to \cite[Proposition 3.1]{DumasRigal}.
\end{rmk}

In Section \ref{sec.gjma} we construct a family of $\kk$-algebras on generators $a,b,c,u$ with relations
\begin{align*}
	& ac=ca+cg, cu=uc, au=ua+ug \\
	& bc=cb+c\gamma, bu=ub+u\gamma, ba = (a+h)b + (h-u)a,
\end{align*}
where $g \in \kk[c]$, $h=cg'$, and $\gamma=h+u+2a$.
We denote such an algebra by $\mM_f$, where $f=cg$.
Note that $\mM_{c^2}$ gives the alternate presentation for $\mj$ in Remark \ref{MJ Alternative}.

The centers of $\mq$ and $\mj$ are each generated by a single degree two element, known as the \textit{quantum determinant}. 
In the case of $\mq$, the quantum determinant is $bc-qad$, and for $\mj$ the quantum determinant is $ad-cb-cd$.
The element $z=gb+(g-u)a-a^2$ is central in $\mM_f$.

If $M=\mq$ or if $M=\mj$ and $z$ is the quantum determinant of $M$, 
then $M/(z-1)$ is a deformation of the coordinate ring of functions on $\SL_2$.
In particular, $M/(z-1)$ is a \textit{noncommutative quadric} \cite{SmithVDB}. 
We can similarly construct such factor rings when $M$ is a GJMA as constructed above.
In this way, we view $\mM_f/(z-1)$ as a sort of nonhomogeneous noncommutative quadric.

\section{Skew polynomial rings and Weyl algebras}
\label{sec.skewpolys}

Let $R$ be a ring.
The first Weyl algebra over $R$, denoted $A_1(R)$, is the overring of $R$ 
with additional generators $x$ and $y$ which commute with $R$ and satisfy $xy=yx+1$.

Let $\sigma$ be an automorphism of $R$ and let $\delta$ be a $\sigma$-derivation, that is,
$\delta:R \rightarrow R$ is a $\kk$-linear map such that $\delta(rs) = \sigma(r)\delta(s) + \delta(r)s$ for all $r,s \in R$. 
The \textit{skew polynomial ring} (or Ore extension) $R[x;\sigma,\delta]$ is defined via the commutation rule 
$x r  = \sigma(r)x + \delta(r)$ for all $r \in R$. 
If $\delta=0$, then we omit it and write $R[x;\sigma]$. 
If $\sigma$ is the identity, then we omit it and write $R[x;\delta]$.
A skew polynomial ring of this last form is called a \textit{differential operator ring} over $R$.

A $\sigma$-derivation $\delta$ of $R$ is said to be 
$\sigma$-\textit{inner} if there exists $\theta \in R$ such that 
$\delta(r)=\theta r-\sigma(r)\theta$ for all $r \in R$. 
In this case, the skew polynomial ring 
$R[x;\sigma,\delta]$ is equivalent to $R[x-\theta;\sigma]$. 
An automorphism $\sigma$ of $R$ is said to be \textit{inner} 
if there exists a unit $\varphi \in R$ such that 
$\varphi\inv r\varphi = \sigma(r)$ for all $r \in R$. 
In this case, the skew polynomial ring $R[x;\sigma,\delta]$ 
is equivalent to $R[\varphi x;\varphi\delta]$.

Let $\sigma_1$ be an automorphism of $\kk[c]$ and 
$\delta_1$ a $\sigma_1$-derivation. 
Write $A=\kk[c][a;\sigma_1,\delta_1]$.
By \cite[Remark 2.1]{AVDVO}, $A$ is one of the following:
\begin{itemize}
	\item Quantum plane: 
	$\qp = \kk\langle a,c \mid ac-qca \rangle$, $q \in \kk^\times$;
	\item Quantum Weyl algebra: 
	$\wa = \kk\langle a,c \mid ac-qca-1 \rangle$, $q \in \kk^\times$;
	\item Differential operator ring: 
	$R_f = \kk\langle a,c \mid ca-ac+f \rangle$, $f \in \kk[c]$.
\end{itemize}
There is no overlap between the different classes except that
the first Weyl algebra $A_1(\kk)$ over $\kk$ is both a quantum Weyl algebra
with $q=1$, and a differential operator ring with $f=1$.
See \cite{GadIso} and \cite{Gad2gen} for further study on
isomorphisms within each class.

Our work will focus heavily on the class of differential operator rings above.
These algebras have been studied extensively by various authors.
For a good overview, see \cite{GW}.
We review some basic facts about these algebras here.

By \cite[Theorem 1.1]{Gad2gen}, $R_f$ is isomorphic to either $A_1(\kk$),
the universal enveloping algebra of the $2$-dimensional solvable Lie algebra ($R_c$), 
the Jordan plane ($R_{c^2}$), or the deformed Jordan plane ($R_{c^2-1}$).
For all $f \in \kk[c]$, the differential operator ring $R_f$ is a
noetherian domain of global and GK dimension two.
In characteristic zero, the center of $R_f$ is $\cnt(R_f)=\kk$.

Denote by $\Frac(R)$ the clasical quotient ring of $R$.
Two rings $R$ and $S$ are said to be \textit{birationally equivalent} 
if $\Frac(R) \iso \Frac(S)$.
We will make use of additional properties for differential operator 
rings from \cite[Proposition 1.8]{DumasRigal} and \cite[Proposition 2.6]{ADinvariants}.
We state these results below using our notation.

\begin{prop}[Alev and Dumas]
\label{prop.diff}
Choose $f,g \in \kk[c]$, $f,g \neq 0$.
\begin{enumerate}
	\item The differential operator ring $R_f$ is birationally equivalent to $R_1=\fwa$.
	\item $R_f \iso R_g$ if and only if there exists 
	$\lambda, \alpha \in \kk^\times$ and $\beta \in \kk$ such that 
	$f(c) = \lambda g(\alpha c + \beta)$.
	\item For any $\alpha, \lambda \in \kk^\times$ and $\beta \in \kk$
	there is an automorphism of $R_f$ such that $f(\alpha c + \beta)=\alpha\lambda f(c)$
	determined by the assignments $a \mapsto \lambda a - g$ and $c \mapsto \alpha c + \beta$.
	Moreover, when $f \notin \kk$ (i.e., $R_f \niso \fwa$) any automorphism of $R_f$ has the above form.
\end{enumerate}
\end{prop}

We conclude this section with automorphisms
of the first Weyl algebra, 
$\fwa = \kk\langle a,c \mid ac=ca+1 \rangle$.
Diximier \cite{dix} gave generators for $\Aut(\fwa)$,
but we only need the limited subgroup of
automorphisms detailed in the following proposition.

\begin{prop}
\label{prop.fwa}
If $\sigma \in \Aut(\fwa)$ satisfies $\sigma(c)=c$, 
then $\sigma(a) = a-g$ for some $g \in \kk[c]$.
\end{prop}

\begin{proof}
Write $\sigma(a)= \sum_{i=0}^n a^i g_i$ for $g_i \in \kk[c]$. Then
\begin{align*}
1 	&= \sigma\left([a,c]\right)
	= \left[ \sigma(a), c \right]
	= \left[  \sum_{i=0}^n a^i g_i, c \right]
	= \sum_{i=0}^n [a^i,c] g_i
	= \sum_{i=0}^n i a^{i-1} g_i.
\end{align*}
Consequently, $g_i=0$ for $i>1$, $g_1=1$, and $g_0$ is arbitrary.
Set $g=-g_0$.
\end{proof}

\section{Involutive skew polynomial extensions}
\label{sec.ispe}

\begin{defn}
Let $R$ be a ring.
We say $S=R[a;\sigma_1,\delta_1][d;\sigma_2;\delta_2]$ 
is an \textbf{involutive skew polynomial extension} (ISPE) of $R$ if 
there exists an involution $\tau \in \Aut(S)$ such that 
$\tau(a)=d$, $\tau(d)=a$, and $\tau(r)=r$ for all $r \in R$.
\end{defn}

\begin{rmk}
One possible generalization of this definition is to requires $S$
to be a {\bf double extension} of $R$ as defined by Zhang and Zhang \cite{zz1,zz2}.
Because our primary interest is in (ordinary) skew polynomial rings,
we do not make that definition here.
\end{rmk}

It is well known that such an $S$ is a noetherian domain if $R$ is.
We give several examples of ISPEs below.

\begin{ex}
When $R=\kk$, $R[a;\sigma_1,\delta_1]$ is a polynomial extension 
and so $\sigma_1=\id_R$ and $\delta_1=0$.
Then $S$ is generated by $a$ and $d$ subject to 
the single relation $0=qad-da+f(a)$ for some $q \in \kk^\times$
and $f(a) \in \kk[a]$.
Applying the involution $\tau$ gives
\[ 0 = qda - ad + f(d) = q(qad+f(a))-ad+f(d) = (q^2-1)ad + (qf(a)+f(d)).\]
Hence $q^2=1$ and $f$ is a constant polynomial.
This implies that either $S=\qp$ with $q=\pm 1$ or $S=\wa$ with $q=-1$.
\end{ex}

Let $S$ be an ISPE of a ring $R$. 
The element $u=d-a$ of $S$ appears in \cite{DumasRigal} 
and is critical to the study of birational equivalence in ISPEs. 
Using the involution $\tau$, one can show 
$\sigma_1(r)=\sigma_2(r)$, $\delta_1(r)=\delta_2(r)$,
and hence $ur=\sigma_1(r)u$ for all $r \in R$. 

\begin{prop}
Let $S$ be an ISPE of a ring $R$.
The element $u=d-a$ is normal in $S$ if and only if
$\delta_2(a)=a^2-\sigma_2(a)a$.
\end{prop}

\begin{proof}
We have already seen that $ur=\sigma_1(r)u$ for all $r \in R$.
Moreover, the $d$-degree of $\delta_2(a)-a^2$ is zero in
\[ (d-a)a = \sigma_2(a)d + \delta_2(a) - a^2.\]
Thus, $(d-a)$ is normal if and only if $(d-a)a=\sigma_2(a)(d-a)$
if and only if $\delta_2(a)-a^2=-\sigma_2(a)a$.
\end{proof}

In this case, we have an alternate presentation of $S$ with relations
\begin{align*}
ar &= \sigma_1(r)a + \delta_1(r) \text{ and }
ur = \sigma_1(r)u \text{ for all } r \in R, \\
ua &= \sigma_2(a)u.
\end{align*}

\begin{prop}
\label{ispe.birat}
Let $S$ be an ISPE of $R$ and suppose the following conditions hold:
\begin{itemize}
\item $R$ is affine and commutative with generators 
$r_1,\hdots,r_n$ and $\sigma_1 = \id_R$.
\item $u$ is normal in $S$.
\item $r_i\inv \delta_1(r_i) = r_j\inv \delta_1(r_j)$ for all $i,j \in \{1,\hdots,n\}$.
\item $a \left(r_i\inv u\right) = \left(r_i\inv u\right) a$.
\end{itemize}
Then $S$ is birationally equivalent to $A_1(R)$.
\end{prop}

\begin{proof}
Let $Q=\Frac(S)$ and let $x,y,t_i$ be the standard generators of $A_1(R)$. 
Let $f=r_1\inv \delta_1(r_1)$.
Because $u$ commutes with $R$, 
the elements $r_i\inv u$ in $Q$ generate an isomorphic copy of $R$. 
We associate these elements with the standard generators $t_i$ of $R$ in $A_1(R)$. 
Hence, by abuse of notation, $f$ may simultaneously be thought of 
as an element of $R$ in the generators $r_i$ and in the $t_i$. 
It now follows by the hypotheses that there are inverse homomorphisms 
$\Phi :Q \rightarrow \Frac(A_1(R))$ and
$\Psi :\Frac(A_1(R))\rightarrow Q$ given by
\begin{align*}
& \Phi (a)=xf,		\Phi(u)=yt_1,	\Phi(r_i)=t_i\inv yt_1, \\
& \Psi (x)=af\inv,	\Psi(y)=r_1,		\Psi(t_i)=r_i\inv u.
\end{align*}
\end{proof}

\begin{ex}
Let $R$ be any affine commutative ring with generators $r_1,\hdots,r_n$.
Let $S$ be the extension of $R$ with relations
$a r_i = r_i a + r_i$, $d r_i = r_i d + r_i$, and $da = (a-1)d + a$.
Then $S$ is an ISPE of $R$.
By Proposition \ref{ispe.birat}, $S$ is birationally equivalent to $A_1(R)$.
\end{ex}

\begin{ex}
Let $R=\kk[c]$, 
$\sigma_1(c)=\sigma_2(c)=c$, 
$\delta_1(c)=\delta_2(c)=c^2$, 
$\sigma_2(a)=(a-c)$, and $\delta_2(a)=ca$. 
Then $S$ is an ISPE. 
In particular, $S$ is the base ring in the skew polynomial construction of $\mj$.
By Proposition \ref{prop.pma-birat}, $S$ is birationally equivalent to $A_1(\kk[t])$.
\end{ex}

\section{Classification of ISPEs over $\kk[c]$}
\label{sec.pmas}

Throughout this section, 
let $A=\kk[c][a;\sigma_1;\delta_1]$ and 
$P=A[d;\sigma_2;\delta_2]$.

\begin{prop}
\label{pma.wa}
If $P$ is an ISPE with $A=\wa$ and $q \neq 1$, 
then $q=-1$ and $P$ has relations $ac+ca=1$, $dc+cd=1$, 
$da+ad=h$ with $h \in \kk[c]$.
\end{prop}

\begin{proof}
We have $A=\kk\langle a,c \mid ac-qca-1 \rangle$.
Let $\sigma=\sigma_2$ and $\delta=\delta_2$.
The existence of the involution $\tau$ implies $dc=qcd+1$. 
In particular, $\sigma(c)=qc$ and $\delta(c)=1$.
Because $\sigma$ is an automorphism of $\wa$, then $\sigma(a)=q\inv a$ \cite[Proposition 1.5]{ADrigid}
giving the final relation
\begin{align}
\label{wa.rel1}
da=q\inv ad + \delta(a).
\end{align}
Applying $\tau$ to \eqref{wa.rel1} gives $ad = q\inv da + \tau(\delta(a))$, or equivalently,
\begin{align}
\label{wa.rel2}
da=q ad - q\tau(\delta(a)).
\end{align} 
Combining \eqref{wa.rel1} and \eqref{wa.rel2} yields
$(q-q\inv)ad = \delta(a) + q\tau(\delta(a))$.
Since the $d$-degree (resp. $a$-degree) of $\delta(a)$ (resp. $\tau(\delta(a))$) is zero,
then $(q-q\inv)ad=0$, implying $q=\pm 1$.

If $q=-1$, then we are left only to determine $\delta(a)$.
The above computation shows that $\delta(a)$ is $\tau$-invariant.
Thus, $\delta(a) \in \kk[c]$.
We need only show that any choice of polynomial in $\kk[c]$ produces a valid $\sigma$-derivation.
\begin{align*}
\delta(ac+ca - 1) 
&= \left( \sigma(a)\delta(c) + \delta(a)c \right) + \left( \sigma(c)\delta(a) + \delta(c)a\right) \\
&= \left( -a + \delta(a)c \right) + \left( -c\delta(a) + a \right)
= \delta(a)c - c\delta(a).
\end{align*}
Thus, $\delta$ is a $\sigma$-derivation if $\delta(a)$ commutes with $c$.
\end{proof}

\begin{rmk}
Set $h=-2$ in Proposition \ref{pma.wa} and let $I$ denote the ideal generated by 
$c^2$, $a^2$, and $d^2$, which are all central elements of $P$ in this case. 
The elements $x_1=a+c+I$, $x_2=c+d+I$, and $x_3=\sqrt{-1/2}(a+d)+I$ of the factor ring $P/I$ satisfy the following:
\[ x_ix_j = \begin{cases}
-x_jx_i & \text{if }  i\neq j \\ 
1 & \text{if }  i=j.
\end{cases}
\]
It is easy to see that $P/I$ is isomorphic to the Clifford algebra of a
$3$-dimensional regular quadratic space over $\kk$. 
(A good reference for Clifford algebras is \cite{Lam}.)
\end{rmk}

We now proceed to the quantum plane case.

\begin{prop}
\label{pma.qp}
Choose $q \in \kk^\times$ such that $q \neq 1$ and let $A=\qp$. 
There are two possible ISPEs
$P=A[d;\sigma_2,\delta_2]$ with common relations 
$ac=qca$, $dc=qcd$, and either $da=ad$ or $da+ad=h$ for some $h \in \kk[c]$.
\end{prop}

\begin{proof}
We have $A=\kk\langle a,c \mid ac-qca \rangle$.
Let $\sigma=\sigma_2$ and $\delta=\delta_2$.
The existence of the involution $\tau$ implies $dc=qcd$
so that $\sigma(c)=qc$ and $\delta(c)=0$.
Moreover, since $\sigma$ is an automorphism of $A$, 
$\sigma(a)=pa$ for some $p\in \kk^\times$ \cite[Proposition 1.4.4]{AlevChamarie}.
We have
\[ \delta(a)=\sum_{j=0}^{m}h_{j}a^{j}\]
for some $h_{0},\ldots,h_m \in k[c]$. 
Applying $\tau$ to the commutation relation $da=pad+\delta(a)$ gives 
\[ ad=pda+\sum_{j=0}^{m}h_{j}d^{j}\]
since $\tau(h_j) = h_j \in k[c]$ for all $j$.
Combining this with $da=pad+\delta(a)$ gives 
\[ ad=p^2ad+p\sum_{j=0}^{m}h_ja^j+\sum_{j=0}^{m}h_jd^j\]
and equating coefficients of $d$ on both sides of the equation gives $p^2=1$.

If $p=1$, then $h_j=0$ for all $j$, and $\delta(a)=0$. 
On the other hand, if $p=-1$, then $h_{j}=0$ for all $j\geq 1$.
\end{proof}

\begin{rmk}
The above arguments yield the same result if we replace $\tau$ with a
sign-changing automorphism $\sigma$ such that $\sigma(a)=-d$ and $\sigma(d)=-a$.
This is reminiscent of the distinction between Lie algebra and Lie superalgebra.
\end{rmk}

Recall the algebra $R_f = \kk\langle a,c \mid ac=ca+f \rangle$ is defined for any $f \in \kk[c]$.
To construct ISPEs in this case, we describe $R_f$
as an Ore extension, $R_f = \kk[c][a;\delta_1]$ with $\delta_1(c)=f$.
An easy induction argument shows $ac^n=c^na+nfc^{n-1}$ for all $n \geq 0$.
This implies $ah = ha+fh'$ and $\delta_1(h)=f \partial_c(h) = fh'$,
where $\partial_c(h) = h'$ is the usual derivative of $h$.

\begin{prop}
\label{prop.pmapres}
Suppose $f\in k[c]$
and $R_f[d;\sigma_2;\delta_2]$ is an ISPE. 
There exists $g \in k[c]$ such that 
$\sigma_2(c)=c$, $\delta_2(c)=f$, 
$\sigma_2(a)=a-g$, and $\delta_2(a)=ga$. 
Thus we may denote $R_f[d;\sigma_2,\delta_2]$ 
by $P(f,g)$ since it depends only on $f$ and $g$.
\end{prop}

\begin{proof}
Applying $\tau$ to $ac=ca+f$ gives $dc=cd+f$ so 
$\sigma_2(c)=c$ and $\delta_2(c)=f$.
We find $\sigma_2(a) = a-g$ by 
Proposition \ref{prop.diff} ($\deg f > 1$),  
Proposition \ref{prop.fwa} ($f \in \kk$)
and \cite[Theorem 2]{dicks} if $f=0$.
We check the formula for $\delta_2(a)$.

Recall $\delta_1(h)=fh'$ for all $h \in \kk[c]$.
A similar argument applies to $\delta_2$ since $\delta_2(c)=\delta_1(c)=f$.
This shows $\delta_2(h)=\delta_1(h)=fh'$ for all $h \in \kk[c]$.
Applying $\delta_2$ to both sides of $ac=ca+f$ and reducing leads to the formula $\delta_2(a)c=c\delta_2(a)+gf$.

We have $\delta_2(a) \in R_f$ so $\delta_2(a)=\sum_{j=0}^m h_ja^j$ for some $h_0,\hdots,h_m \in \kk[c]$.
We substitute this formula into $\delta_2(a)c=c\delta_2(a)+gf$.
\begin{align*}
	gf 	&= \delta_2(a)c-c\delta_2(a)
		= \sum_{j=0}^{m} h_{j} \left(a^jc - ca^j \right)
		= \sum_{j=0}^m h_{j} \left(\left(ca^j + \sum_{i=1}^j \binom{j}{i} \left(\delta_1(c)\right)^i a^{j-i}\right)-ca^j \right)  \\
		&= \sum_{j=1}^{m}\sum_{i=1}^{j}h_{j} \binom{j}{i} \left(\delta_1(c)\right)^i a^{j-i}
		= \sum_{x=0}^{m-1} \left(\sum_{y=1}^{m-x}\binom{x+y}{y}h_{x+y}(\delta_1(c))^y \right) a^{x}.
\end{align*}
We have $gf\in \kk[c]$ so $gf=\sum_{y=1}^m h_y \left(\delta_1(c)\right)^y$ and 
$\sum_{y=1}^{m-x}\binom{x+y}{y} h_{x+y}\left(\delta_1(c)\right)^y=0$ for all $x$, $1\leq x\leq m-1$. 
Setting $x=m-1$ gives $mh_m \delta_1(c) a^{m-1}=0$ so $h_m=0$. 
Setting $x=m-2$ gives $(m-1)h_{m-1}\delta_1(c)=0$ so $h_{m-1}=0$. 
Continuing in this fashion, we find $h_{\ell}=0$ for all $\ell=2,\ldots,m$. 
This gives $gf=h_1 \delta_1(c)=h_1f$ so $g=h_1$. 
Moreover, $\delta_2(a)=h_0+h_1a$ so $\delta_2(a)=h_0+ga$.

Our arguments have lead us to the relation $da=(a-g)d+h_{0}+ga$. 
Applying $\tau$ and reducing gives $h_0=0$.
We are left with $da=(a-g)d+ga$ and $\delta_2(a)=ga$, as desired.
\end{proof}

\begin{ex}
The base ring of the Jordanian matrix algebra $\mj$ is the ISPE $P(c^2,c)$.
\end{ex}

\section{The algebras $P(f,g)$}
\label{sec.pmadiff}

In this section we consider the properties of the ISPEs over $\kk[c]$
relative to the differential operator rings $R_f$,
including presentations, prime spectrum, and birational equivalence.

\begin{rmk}
\label{rem.pres}
An easy check shows $\delta_2(c)=ac-\sigma_2(c)a$ and $\delta_2(a)=a(a)-\sigma_2(a)a$ in $P(f,g)$.
Thus, $\delta_2$ is the inner $\sigma_2$-derivation induced by $a$.
\end{rmk}

\begin{rmk}
\label{Presenting P}
The presentation of $\mj$ determined by setting $u=d-a$ extends to alternative presentations of $P(f,g)$ (see Remark \ref{MJ Alternative}). 
We have $du=u(d+g)$, $ac=ca+f$, $uc=cu$, and $au=ua+gu$.

\begin{description}
\item[Presentation 1] $P(f,g)=\kk[c][a;\delta_1][u;\sigma_2]$.
\item[Presentation 2] $P(f,g)=\kk[c,u][a;D]$ where $D$ is the derivation $f\partial_c + gu\partial_u$.
\end{description}
\end{rmk}

A (two-sided) ideal $I$ of $R_f$ is prime if and only if $I \cap \kk[c]$ is $\delta$-invariant.
That is, if $J=I \cap \kk[c]$, then $\delta_1(J) \subseteq J$.
Thus, if $I$ is a prime ideal of $R_f$, then $I$ is generated by an irreducible factor of $f$.
Hence, since $\kk$ is algebraically closed, $I=(c-\alpha)$ for some $\alpha \in \kk$ such that $f(\alpha)=0$.

\begin{prop}
\label{pma.props}
Choose $f,g \in \kk[c]$ and set $P=P(f,g)$.
The following elements of $P$ are normal: 
$f$, $u=d-a$, and $c-\alpha$ for any $\alpha \in \kk$ such that $f(\alpha)=0$.
Moreover, the ideals generated by (i) $u$, (ii) $c-\alpha$, and (iii) 
both $u$ and $c-\alpha$ are prime in $P$.
\end{prop}

\begin{proof}
Denote the ideals described in $(i)$, $(ii)$, and $(iii)$ of (1) by $I_1$, $I_2$, and $I_3$, respectively.

We showed $u$ is normal in Remark \ref{Presenting P}.
Suppose $f \notin \kk$ and 
note that $f$ is a product of its linear factors since $\kk$ is algebraically closed.
If $\alpha \in \kk$ such that $f(\alpha)=0$, then there exists $f_0 \in \kk[c]$ such that $f=(c-\alpha)f_0$ and
$c(c-\alpha)=(c-\alpha)c$, $a(c-\alpha)=(c-\alpha)(a+f_0)$, and $d(c-\alpha)=(c-\alpha)(d+f_0)$ so $c-\alpha$ is normal.

Then $P/I_1 \iso R_f$ and $P/I_3 \iso \kk[a]$, so $I_1$ and $I_3$ are prime.
Moreover, 
\[
P/I_2 \iso 
\begin{cases}
	k[a,d] 	& \text{if } g(\alpha)=0 \\ 
	U(L) 	& \text{if } g(\alpha)\neq 0,
\end{cases}
\]
where $U(L)$ is the universal enveloping algebra of the two-dimensional solvable Lie algebra $L$. 
Hence $I_2$ is also prime.
\end{proof}

\begin{defn}
\label{def.simp}
Let $f,g\in \kk[c]$, $p_1,\ldots,p_k$ 
the irreducible factors of $f$ and $f_i=p_i\inv f$.
We say \textbf{$g$ is a local reduction of $f$} if there exists integers 
$n,m_1,\hdots,m_k$ ($n \neq 0$) such that 
\[ g = \begin{cases}
	0 & \text{ if } f \in \kk \\
	-\frac{1}{n}\sum_{i=1}^{k}m_{i}f_{i} & \text{otherwise}.
	\end{cases}\]
\end{defn}

\begin{prop}
\label{pma.simp}
Let $P=P(f,g)$ and let $\mC$ be the Ore set generated by powers 
of the irreducible factors $p_1,\hdots,p_k$ of $f$.
Let $\mR=R_f\mC\inv$.
Then $Q=\mR[u^{\pm 1};\sigma_2]$ is simple if and only if 
$f \in \kk$ or $g$ is not a local reduction of $f$.
\end{prop}

\begin{proof}
It follows from Proposition \ref{prop.diff} that $\mR$ is a simple ring. 
Thus by \cite[Theorem 1.8.5]{McRob}, 
$Q$ is simple if and only if no power of $\sigma_2$ is inner. 
Set $\sigma =\sigma_2$ and set $f_i=p_i\inv f$ for $i=1,\hdots,k$.
Set $f_i=0$ if $f \in \kk$.

Suppose $\sigma^n$ is an inner automorphism of $\mR$ for some integer $n$.
Then there exists a unit $\eta \in \mR$ such that
$\eta\inv r \eta = \sigma^n(r)$ for all $r \in \mR$.
Since $\eta$ is a unit, then (up to a scalar) 
$\eta = p_1^{m_1}\cdots p_k^{m_k}$ for some $m_i \in \ZZ$.
Recall, 	\[ p_i\inv a p_i = p_i\inv (p_i a + f) = a + f_i.\]
By Proposition \ref{prop.pmapres}, $\sigma(a)=a-g$ and
\[	a-ng = \sigma^n(a) = \eta^{-1}a\eta = a + \sum_{i=1}^k m_i f_i. \]
Thus, $g = -\frac{1}{n} \sum_{i=1}^k m_i f_i$.
The converse is similar.
\end{proof}

\begin{cor}
\label{pma.primes1}
If $f,g \in \kk[c]$ are such that $g$ is not a local reduction of $f$,
then the height one primes of $P$ are of the form $(u)$ and
$(c-\alpha)$ for $\alpha \in \kk$ with $f(\alpha)=0$.
Moreover, $P/(c-\alpha) \iso \kk[a,u]$ and $P/(u) \iso R_f$.
\end{cor}

In case $g$ is not a local reduction of $f$, 
the finite-dimensional representation theory of $P=P(f,g)$ is easy to describe in light of 
Corollary \ref{pma.primes1}.
In particular, for every $\alpha,\beta \in \kk$ with $f(\alpha)=0$, 
there exists a $1$-dimensional (irreducible) $P$-module $M=\{m\}$ such that $u.m=0$, $c.m=\alpha m$ and $a.m=\beta m$. 
On the other hand, for every irreducible polynomial $\Omega \in \kk[a,u]$ of degree $n$, 
there exists an irreducible $n$-dimensional representation.

\begin{prop}
\label{pma.props2}
Choose $f,g \in \kk[c]$ and set $P=P(f,g)$.
The algebra $P$ is a noetherian domain with $\GKdim P = 3$ and
	\[ \gldim P = \begin{cases} 
		2 & \text{if } f \in \kk \\
		3 & \text{otherwise}.
		\end{cases}
	\]
\end{prop}

\begin{proof}
It follows easily that $P$ is a noetherian domain from the skew polynomial ring construction.
The ring $\kk[c,u]$ has GK and global dimension two.
Using Presentation 2, the result on GK dimension for $P$ 
now follows by \cite[Corollary 8.2.11]{McRob}.

By \cite[Theorem 7.10.3]{McRob}, $\gldim P = 3$ if $P$ has a $D$-stable prime
of height two and otherwise $\gldim P = 2$.

Suppose $f \in \kk$. If $g \neq 0$, then by Proposition \ref{pma.simp},
$(u)$ is the unique maximal ideal of $P$ and it has height one.
If $g=0$, then $P \iso \fwa[u]$ and by the simplicity of $\fwa$,
the maximal ideals are of the form $(u-\beta)$, $\beta \in \kk$,
and they each have height one.

Now suppose $f \notin \kk$. 
As in Proposition \ref{pma.simp}, let $\alpha \in \kk$ such that 
$f(\alpha)=0$ and set $f_0=(c-\alpha)\inv f$.
Since $f\in I_3$, the calculation below shows $D(ur+(c-\alpha)s)\in I_3$ for $r,s\in P$. 
\begin{align*}
D(ur+(c-\alpha)s) 
	&= uD(r) + D(u)r + (c-\alpha)D(s) + D(c-\alpha)s \\
	&= u\left(D(r) + gr\right) + (c-\alpha)\left(D(s)+f_0s\right).
\end{align*}
Therefore, $I_3$ is a $D$-stable prime of height two in $P$ and
we may conclude $P$ has global dimension three.
\end{proof}

The base ring of the Jordanian matrix ring $\mj$ is $P(c^2,c)$ and, 
in this case, $g$ is a local reduction of $f$.
More generally, if $f=cg$, then $g$ is a local reduction of $f$ 
since $c$ is an irreducible factor of $f$ and $g = c\inv f$.
In this case the prime ideal spectrum must be more complicated 
since $Q$ (as defined in Proposition \ref{pma.simp}) is not simple.
In the ISPE $P(cg,g)$, the element $u-\beta c$ is normal for any $\beta \in \kk$.
We wish to show that the remaining ideals are those of the form $(u-\beta c)$.

\begin{prop}
\label{prop.cntP} 
Let $P=P(f,g)$ for some $f,g\in \kk[c]$ such that $f$ is monic. 
Set $\mP=P(f,g)\mC\inv$, where $\mC$ is the Ore set generated by 
powers of the monic irreducible factors $p_1,\hdots,p_k$ of $f$.
If $g$ is a local reduction of $f$, then set
$g=-\frac{1}{n}\sum_{i=1}^{k}m_{i}f_{i}$ with 
$m_1,\hdots,m_k\in \ZZ$, $n \in \NN$, and $f_i = p_i{\inv}f$. 
\begin{enumerate}
\item If $\cnt(\mP) \neq \kk$, then $g$ is not a local reduction of $f$.
\item If $g$ is a local reduction of $f$, 
then $\cnt(\mP) = \kk[v]$ with $v=-\left( p_{1}^{-m_{1}}\cdots p_{k}^{-m_{k}}\right) u^{n}$.
\item The center of $P(f,g)$ is $\kk$.
\end{enumerate}
\end{prop}

\begin{proof}
We only need to prove (2) since (1) follows from
Proposition \ref{pma.simp} and 
(3) follows from (1) and (2).

A direct calculation shows $ap_{i}=p_{i}\left( a+f_{i}\right) $ for all $i\leq k$. Moreover,
\[ a\left( p_1^{-m_1}\cdots p_k^{-m_k}\right) 
= \left( p_1^{-m_1}\cdots p_k^{-m_k}\right)  \left(a+\sum_{i=1}^{k}-m_{i}f_{i}\right)\]
and 
\[ au^n = u^n \left( a+ng\right) = u^n\left(a+\sum_{i=1}^{k}-m_{i}f_i\right).\]
Putting these two equations together gives 
\[a\left( p_1^{-m_1}\cdots p_k^{-m_k}\right) u^n = \left( p_1^{-m_1}\cdots p_k^{-m_k}\right) u^{n}a,\]
hence $\left( p_1^{-m_1}\cdots p_k^{-m_k}\right) u^n$ is central. 
This proves $v$ is central. 
To finish the proof of (2), we must show any central element of 
$\mP$ is a linear combination of powers of $v$.

Given $\omega \in \cnt(\mP)$ we may write 
$\omega = \sum r_{i}a^{i}$ with $r_{i}\in \kk[c,u]\mC\inv$. 
Since $\omega$ is central, then
\[ 0=[\omega ,c]=\sum r_{i}[a^{i},c]=\sum r_{i}\left( (a+g)^{i}-a^{i})\right).\]
Hence, the $a$-degree of $\omega$ is zero. 
We assume $\omega \notin \kk$ and write 
\[ \omega =\sum_{i=0}^{m}\omega_{i}u^{i}. \]
with $\omega_0,\hdots,\omega_m\in \kk[c]\mC\inv$
such that $\omega_m \neq 0$, $m\geq 1$. 
Since $a\omega =\omega a$, then $D\omega=0$.
Thus,
\begin{align*}
0  
&= D\omega 
 = \sum_{i=0}^m \left( \left( D\omega_i\right) u^i + \omega_i\left(Du^i\right) \right)
 = \sum_{i=0}^m \left( f\left( \partial_{c}\omega_i \right) u^i + ig\omega_i u^i\right)
 = \sum_{i=0}^m \left( f\left( \partial_{c}\omega_i \right) +ig\omega_i \right) u^i.
\end{align*}
According to the above calculation, 
$f\left( \partial_{c}\omega_{i}\right)+ig\omega_{i}=0$ for $i=0,\hdots,m$. 
We may use separation of variables to solve this differential equation. 
This gives $\omega_{i} = \lambda_{i}\left( p_1^{-m_1}\cdots p_k^{-m_k}\right)^{i/n}$ for some $\lambda_i\in \kk $. 
If $\lambda_i\neq 0$, then the power $i/n$ must be a nonnegative integer since $\omega_i\in \kk[c]\mC\inv$. 
This finishes the proof of (2) since 
\[ \omega = \sum_{i=0}^{m}\lambda_{i}\left( p_1^{-m_1}\cdots p_k^{-m_k}\right)^{i/n}u^{i} \
= \sum_{i=0}^{m}(-1)^{i/n}\lambda_{i}v^{i/n}\in \kk[v].\]
\end{proof}

\begin{rmk}
In the special case $f=cg$, we have $v=c^{-1}u$.
\end{rmk}

Fix $f \in \kk[c]$.
If $f \notin \kk$, then up to isomorphism of $P(f,g)$, 
we may choose $f$ such that $f(0)=0$.
\emph{For the remainder of this section, 
we set $g=c\inv f$ and let $P=P(cg,g)$ when
$f \notin \kk$.}
This hypothesis implies that $g$ is a local reduction of $f$.
Hence, $\cnt(\mP)=\kk[c\inv u]$. Set $v=c\inv u$.
If $f \in \kk$, we set $f=1$ and $g=0$.

\begin{prop}
\label{pma.primes2}
If $f \notin \kk$, then 
the height one primes of $P$ are of the form
$(c-\alpha)$ and $(u-\beta c)$ for $\alpha, \beta \in \kk$ with
$f(\alpha)=0$ and $\beta$ arbitrary.
Moreover, $P/(c-\alpha)P \iso \kk[a,u]$ and $P/(u-\beta c)P \iso R_f$.
\end{prop}

\begin{proof}
Since $\mR$ is simple, then by \cite[Corollary 2.3]{LM}
the prime ideals of $P$ are in 1-1 correspondence with $\cnt(\mP)=\kk[v]$.
Hence, the remaining prime ideals are generated in $\mP$ by $v-\beta$ and
$(v-\beta) \cap P = (u-\beta c)$.
\end{proof}

There is another way to view the above proposition.
Note that $P[g\inv]$ is isomorphic to the universal enveloping algebra 
of the $3$-dimensional Lie algebra $L$
on generators $x,y,z$ subject to the relations $[x,y]=y$, $[x,z]=z$, and $[y,z]=0$.
Hence, the prime ideals of $P$ disjoint from irreducible factors of $g$ are in $1-1$
correspondence with the prime ideals of $U(L)$.

\begin{prop}
\label{prop.pmauto}
Suppose $f \notin \kk$.
For any $\alpha,\lambda,\mu \in \kk^\times$, $\eta \in \kk$, and $h \in \kk[u,c]$
there is an automorphism $\pi$ of $P$ such that $\pi(g)=\lambda g$ determined by the assignments
\begin{align*}
	a \mapsto \lambda a + h, \;\;
	c \mapsto \ep c, \;\; \text{and }
	u \mapsto \mu u + \eta c.
\end{align*}
Moreover, any automorphism of $P$ has the above form.
\end{prop}

\begin{proof}
Let $\pi \in \Aut(P)$.
By Proposition \ref{pma.primes2}, $\pi$ fixes the ideals $(c-\alpha)$ and $(u-\beta c)$.
Hence $\pi(c)=\ep c + \kappa$ for some $\ep \in \kk^\times$ and $\kappa \in \kk$ with $f(-\ep\inv\kappa)=0$.
Similarly, $\pi(u)=\mu u+\eta c$ for some $\mu \in \kk^\times$ and $\eta \in \kk$.
A degree counting argument now shows that $\pi(a)=\lambda a + h$ for some $\lambda \in \kk^\times$ and $h \in \kk[c,u]$.
To complete the proof we must show $\kappa = 0$, $\pi(g)=\lambda g$, and $\pi(f)=\ep\lambda f$.

Applying $\pi$ to $au=ua+ug$ gives $\lambda\mu gu + \lambda\eta f = \mu\pi(g)u+\eta\pi(g)c$.
Equating coefficients of $u$ gives $\lambda \mu gu=\mu \pi(g)u$ and $\lambda \eta f=\eta \pi(g)c$. 
Thus $\pi(g)=\lambda g$. 
Applying $\pi$ to $ac=ca+f$ and reducing gives $\pi(f)=\ep \lambda f$. 
Combining this with $f=cg$ and $\pi(g)=\lambda g$ gives $\kappa =0$.
\end{proof}

The following proposition is now a consequence of Proposition \ref{ispe.birat}.	
	
\begin{prop}
\label{prop.pma-birat}
The ISPE $P$ is birationally equivalent to $A_1(\kk[t])$.
\end{prop}

\section{Generalized Jordanian Matrix Algebras}
\label{sec.gjma}

\begin{defn}
\label{defn.matalg}
We say $M=P(f,g)[b;\sigma_3,\delta_3]$ is a \textit{Generalized Jordanian Matrix Algebra} (GJMA) 
if it is birationally equivalent to $\mj$ and the involution $\tau$ extends to $M$ with $\tau(b)=b$.
\end{defn}

\begin{rmk}
By \cite[Proposition 1.8]{DumasRigal}, $\mj$ is birationally equivalent to $A_1(\kk[s,t])$.
Thus, $P(f,g)[b;\sigma_3,\delta_3]$ is a GJMA if and only if it
is birationally equivalent to $A_1(\kk[s,t])$ and $\tau$ extends appropriately.
\end{rmk}

\begin{ex}
Since the ISPE $P(1,0) \iso A_1(\kk[t])$, 
then it follows that $A_1(\kk[s,t]) \iso P(1,0)[b]$ is itself a GJMA.
\end{ex}

In this section, we construct a family of GJMAs and study their properties.
We fix some notation throughout this section.
Let $f, g \in \kk[c]$, $f \notin \kk$, 
such that $g$ is a local reduction of $f$ and set $P=P(f,g)$.
Let $v$ be the central element defined in Proposition \ref{prop.cntP}.
Let $\mC$ be the Ore set generated by powers of irreducible factors of $f$
and $\mU$ the Ore set generated by $\mC$ along with powers of $u-\beta c$ for $\beta \in \kk$.
Let $\mP=P\mC\inv$ and $\sP=P\mU\inv$.
To simplify notation, we write $\sigma=\sigma_3$ and $\delta=\delta_3$.
In order to construct GJMAs $M=P[b;\sigma,\delta]$ we consider restrictions on $\sigma$ and $\delta$.
Let $\mM=M\mC\inv$ and $\sM=M\mU\inv$.
We rely on comparisons with $\mj$ and key characteristics of $P$ outlined in Section \ref{sec.pmadiff}
to identify additional restrictions. 

\begin{prop}
\label{prop.gjmain}
If $\sigma$ is an inner automorphism of $\sP$ and
$\delta$ is an inner $\sigma$-derivation, 
then $M= P[b;\sigma,\delta]$ is birationally equivalent to $\mj$.
\end{prop}

\begin{proof}
By the hypothesis, there exist $\varphi \in \mC$ and $\theta \in \sP$ such that
$\sigma(r) = \varphi\inv r \varphi$ and $\delta(r) = \theta r-\sigma(r)\theta$ for all $r \in \sP$.
In particular, the element $z=\varphi(b-\theta)$ is central in $\sM$.
Moreover, $v=c\inv u$ is central in $\mP$ by Proposition \ref{prop.cntP}.
Thus, it is clear $\{a,c,v,z\}$ is a $\kk$-algebra generating set of $\sM$.
Moreover, there are inverse homomorphisms $\Phi:\sM \rightarrow \Frac(A_1(\kk[s,t]))$ and 
$\Psi:\Frac(A_1(\kk[s,t])) \rightarrow \sM$ given by
\begin{align*}
	&\Phi(a) = xf(y),  		\Phi(c)=y, 		\Phi(v)=yt,	\Phi(z)=s, \\
	&\Psi(x) = af\inv(c), 	\Psi(y) = c,	\Phi(t)=v,	\Phi(s)=z.
\end{align*}
Thus, $M$ is birationally equivalent to $\mj$.
\end{proof}

\begin{ex}
Let $P=P(c,1)$ and consider the algebra $G=P[b;\sigma,\delta]$ given by 
\begin{align*}
\sigma(a) &= c\inv ac = a+1, \;\; 
\sigma(u) = u, \;\; 
\sigma(c)=c, \\
\delta(a) 
	&= -(u/2+a)a-(a+1)(-(u/2+a)) 
	= (-ua/2 - a^2) + (au/2 + a^2 + u/2 + a) = u+a, \\
\delta(u) 
	&= \tau \left( \delta(a) \right) - \delta(a) 
	= -u+u+a-(u+a) =-u, \\
\delta(c) 
	&= \theta c-\sigma(c)\theta 
	= -(u/2+a)c-c(-(u/2+a)) = -ac+ca=-c.
\end{align*}
Thus, the relations for $G$ may be given as those for $P(c,1)$ along with
\begin{align*}
	bc=cb-c, \;\;
	bu=ub-u, \;\;
	ba=(a+1)b + (u+a).
\end{align*}
The element $z=2(cb+ca)+uc$ is central in $G$.
It is not difficult to show that $G$ satisfies 
Proposition \ref{prop.gjmain} and that 
$\tau(b)=b$ extends to an involution of $G$.
Hence, $G$ is a GJMA.
\end{ex}

We now restrict to the case $f=cg$ so that $v=c\inv u$.

\begin{prop}
\label{gjma.constr}
Choose $f \in \kk[c]$, $f=cg$, and let $P=P(f,g)$.
Choose $p$ an irreducible factor of $f$ and $\theta \in P$.
Define $G(f,p,\theta):=P[b;\sigma,\delta]$ by
$\sigma(x) = p\inv x p$ and
$\delta(x) = \theta x - \sigma(x) \theta$ for $x=a,u,c$.
Then $G(f,p,\theta)$ is a GJMA if and only if
$\delta(u) = \tau \left( \delta(a) \right) - \delta(a)$ and
$\delta(c) = \tau(\delta(c))$.
\end{prop}

\begin{proof}
By construction, $\sigma$ is an inner automorphism of $P$.
Note that $p\inv c p = c$, $p\inv u p = u$, and
$p\inv a p = p\inv p'f\in P$ since $p$ is a factor of $f$.
Moreover, $\delta$ is the inner $\sigma$-derivation of $P$ 
determined by $\theta$.

By Proposition \ref{prop.gjmain}, $G(f,p,\theta)$
is a GJMA if $\tau$ extends to an involution with $\tau(b)=b$.
We must show $\tau$ is a homomorphism of $G$.
It will then follow that $\tau$ is an involution of $G$.

Applying $\tau$ to the identity 
$0=\sigma(a)b-ba+\delta(a)$ gives
\begin{align*}
0 	&= \left( p\inv (u+a) p \right) b - b(u+a) + \tau(\delta(a)) \\
	&= \left( \sigma(a)b - ba \right) + \left( \sigma(u)b-bu \right) + \tau(\delta(a)) \\
	&= \sigma(u)b - bu + \left( \tau(\delta(a)) - \delta(a) \right).
\end{align*}
This identity is fixed by $\tau$ and hence 
holds if and only if $\delta(u)$ satisfies the hypothesis.
Repeating with $bc=cb+\delta(c)$ gives 
$\tau(\delta(c)) = \delta(c)$.
\end{proof}

\begin{prop}
Suppose $G=G(f,p,\theta)$ is a GJMA.
We localize $G$ by localizing the base ring $P$
and extending $\sigma$ and $\delta$ appropriately.
We write $\mG = \mP[b;\sigma,\delta]$.
Then $v=c\inv u$ is central in $\mG$ if
and only if $\delta(c) = cu\inv\delta(u)$.
\end{prop}

\begin{proof}
By Proposition \ref{prop.cntP}, $v=c\inv u$ is central in $\mP$.
We have
\begin{align*}
bv &= \left( c{\inv}b-c{\inv}\delta(c)c{\inv}\right) u 
	=	c{\inv}\left( ub+\delta(u) \right) - c\inv\delta(c) c{\inv}u
	= 	vb + c\inv \left( \delta(u) -\delta(c)v\right).
\end{align*}
Hence, $v$ is central in $\mG$ if and only if $\delta(u) =\delta(c)v$
or, equivalently, $c\delta(u) = u\delta(c)$.
Because $c$ and $u$ commute, then 
$\delta(c)=cu\inv\delta(u)$.
This formula must agree with $\theta c - c\theta$.
Recall $\delta(u) = \tau\left(\delta(a)\right) - \delta(a)$ and $\delta(u)=u\gamma$,
hence $\gamma = u\inv\left(\tau\left(\delta(a)\right)-\delta(a)\right)$
so
$\delta(c) = cu\inv\delta(u) = cu\inv\left( \tau \left( \delta(a)\right) -\delta(a) \right)$.
\end{proof}

Suppose $\deg g \geq 1$. Set $h=g'c$ and $\theta=g\inv\left(a^2+(u-g)a\right)$.
Then we have
\begin{align*}
\theta(a)-(a+h)\theta
	&= g\inv \left[ \left(a^2+(u-g)a\right)a - a\left(a^2+(u-g)a\right) \right] \\
	&= g\inv \left[ (ua-ga)a-(ua+ug-ga-g'f)a \right] \\
	&= g\inv \left[ -(ug-g'f)a \right] = (h-u)a, \\
\theta(c)-(c)\theta
	&= g\inv \left[ \left(a^2+(u-g)a\right)c - c\left(a^2+(u-g)a\right) \right] \\
	&= g\inv \left[ \left(a+(u-g)\right)c(a+g) - c\left(a^2+(u-g)a\right) \right] \\
	&= cg\inv \left[ \left((a+g)+(u-g)\right)(a+g) - \left(a^2+(u-g)a\right) \right] \\
	&= cg\inv \left[ \left(a^2+ag+ua+ug\right) - \left(a^2+(u-g)a\right) \right] \\
	&= cg\inv \left[ ga+g'f+ug + ga  \right] = c \left( h+u+2a \right).
\end{align*}
The computation for $u$ is similar.

\begin{prop}
The ring $G\left(f,g,g\inv\left(a^2+(u-g)a\right)\right)$ is a GJMA
which we denote by $\mG_f$.
Moreover, $\mG_{c^2} = \mj$.
\end{prop}

\begin{proof}
We have $g\inv a g = a+h$. Set $\gamma = (h+u+2a)$.
By Proposition \ref{gjma.constr},
it suffices to show the following,
\begin{align*}
\tau \left( \delta(a) \right) - \delta(a)
	&= \tau\left((h-u)a)\right) - (h-u)a
	= (h+u)(u+a)-(h-u)a \\
	&= (hu+ha+u^2+ua)-(ha-ua)
	= hu+u^2+2ua = u\gamma = \delta(u).
\end{align*}
and $cu\inv\delta(u) = c\gamma = \delta(c)$.
\end{proof}

The relations for $\mG_f$ can be shown to be
\begin{align*}
	& ac=ca+f, au=ua+ug, cu=uc, \\
	& bc=cb+c\gamma, bu=ub+u\gamma, ba = (a+h)b + (h-u)a.
\end{align*}
Explicitly, $\sigma$ and $\delta$ are given by
\begin{align*}
	& \sigma(c)=c, \sigma(u)=u, \sigma(a)=a+h, \\
	& \delta(c)=c\gamma, \delta(u)=u\gamma, \delta(a)=(h-u)a.
\end{align*}
By construction, the element $p(b-\theta)$ 
lies in the center of any $G(f,p,\theta)$.
The analog of the determinant in the GJMA $G_f$ 
is the central element $z=gb+(g-u)a-a^2$.

We may also write the presentation in terms of the standard generators $a,b,c,d$.
In this case, $\gamma=h+a+d$ and we have relations
\begin{align*}
	& ac=ca+f, dc=cd+f, da=(a-g)d+g(a-h), \\
	& bc=cb+c\gamma, ba = (a+h)b + (h-u)a, bd = (d+h)b + (h-u)d.
\end{align*}

\begin{prop}
\label{ma.dim}
The algebra $\mG_f$ is a noetherian domain with $\gldim \mG_f = \GKdim \mG_f = 4$.
Moreover, the center of $\mG_f$ is $\kk[z]$.
\end{prop}

\begin{proof}
That $\mG_f$ is a noetherian domain follows from the skew polynomial construction.

The statement on the center follows in similar fashion to \cite[Proposition 1.8]{DumasRigal}.
In $\mG_f\mC\inv$ we have $b=g\inv\left(z-(g-u)a+a^2\right)$.
Let $R=\kk[z,v]\mC\inv$, 
then $\mG_f\mC\inv$ is isomorphic to $S=R[a,D]$ where 
$D$ is extended from Presentation 2 of Remark \ref{Presenting P}.
In particular, $D(z)=D(v)=0$.
In $S$, $[c\inv,a]=gc\inv$ and an induction argument shows that
$[c\inv,a^n]=nga^{n-1} + \text{(lower $a$-degree terms)}$.
If $s=\sum_{i=0}^n r_i a^i \in S$ with $r_i \in R$ for all $i$, then
$[c\inv,s]=\sum_{i=0}^n r_i [c\inv,a^i]=r_n nga^{n-1} + \text{(lower $a$-degree terms)}$.
Hence, if $s \in \cnt(S)$, then $s \in R$. 
Since $D(s)=[a,s]=0$, then $s \in \kk[z,v]$.
Thus, $\cnt(S)=\kk[z,v]$ and $\cnt(\mG_f)=\cnt(S) \cap \cnt(\mG_f) = \kk[z]$.

By \cite[Theorem 7.5.3 (i)]{McRob} and Proposition \ref{pma.props}, 
$\gldim \mG_f \leq 1+ \gldim P = 4$.
On the other hand, $P[z]$ is faithfully flat as an $\mG_f$-module.
Thus, by \cite[Theorem 7.2.6]{McRob} $\gldim \mG_f \geq \gldim P[z] = 4$.

Since $V=\{a,c,u\}$ is a finite-dimensional generating subspace of $P$ and $\sigma(V) \subseteq V$, 
then by \cite[Lemma 2.2]{HuhKim} and Proposition \ref{pma.props}, $\GKdim \mG_f = 1+\GKdim P = 4$.
\end{proof}

\begin{prop}
\label{jma.primes}
The height one prime ideals of $\mG_f$ are $(c)$, $(u)$, 
and $(z-\xi)$ for $\xi \in \kk^\times$.
\end{prop}

\begin{proof}
The height one prime ideals of $P$ are principally 
generated by $c-\alpha$ and $u-\beta c$ 
where $\alpha, \beta \in \kk$ with $f(\alpha)=0$ 
and $\beta$ arbitrary (Proposition \ref{pma.primes2}).
Then $\sP$ is simple by Lemma \ref{pma.primes2}.

In $\mG_f$, $c-\alpha$ (resp. $u-\beta c$) 
is a normal element if and only $\alpha=0$ (resp. $\beta=0$).
Hence, the height one prime ideals of $\mG_f$ 
that have nonzero intersection with $P$ are $(u)$ and $(c)$.
By Proposition \ref{ma.dim}, the center of $\mG_f$ is $\kk[z]$.
Hence, the ideals of $\sP$ which lie over $0$ in $P$ are of the form 
$(z-\xi)$ for $\xi \in \kk^\times$ \cite[Corollary 2.3]{LM}.
\end{proof}

\begin{rmk}
We think of $S=\mG_f/(z-1)$ as an analog of $\SL_J(2)$, 
the Jordanian deformation of $\SL(2)$.
Because $(z-1)$ is irreducible in $\mG_f$, 
$(z-1)$ is a completely prime ideal.
That is, $S$ is an integral domain.
\end{rmk}


\bibliographystyle{plain}

\end{document}